\documentclass[12pt,a4paper,twoside]{amsart}

\newcommand*{\R}{\mathbf{R}}

\newcommand*{\dhil}[2]{\ensuremath{d_{H}(#1,#2)}}
\newcommand*{\dhillam}[3]{\ensuremath{d_{H,#1}(#2,#3)}}
\newcommand*{\dthom}[2]{\ensuremath{d_{T}(#1,#2)}}
\newcommand*{\dthomlam}[3]{\ensuremath{d_{T,#1}(#2,#3)}}
\newcommand*{\driem}[2]{\ensuremath{d_{R}(#1,#2)}}

\newcommand*{\pp}{\ensuremath{\mathcal{P}}}

\newcommand*{\normdiam}[1]{\ensuremath{|#1|_{\sigma}}}

\newcommand*{\auto}[1]{\ensuremath{\mathop{\text{Aut}}(#1)}}
\newcommand*{\isom}[1]{\ensuremath{\mathop{\text{Iso}}(#1)}}
\newcommand*{\M}[2]{\ensuremath{\mathop{\text{M}}(#1,#2)}}
\newcommand*{\MC}[3]{\ensuremath{\mathop{\text{M}_{#1}}(#2,#3)}}
\newcommand*{\Trace}[1]{\ensuremath{\mathop{\text{Tr}}(#1)}}
\newcommand*{\trace}[1]{\ensuremath{\mathop{\text{Tr}}(#1)}}
\newcommand*{\quadr}[1]{\ensuremath{\mathop{\text{P}}(#1)}}
\newcommand*{\idem}[1]{\ensuremath{\mathcal{\text{P}}(#1)}}
\newcommand*{\idemrank}[2]{\ensuremath{\mathcal{P}_{#1}(#2)}}

\newcommand*{\specmean}[2]{\ensuremath{#1\mu #2}}
\newcommand*{\Sigm}[2]{\ensuremath{\mathfrak{S}(#1,#2)}}
\newcommand*{\Sigmun}[2]{\ensuremath{\mathfrak{S}_1(#1,#2)}}
\newcommand*{\Sigmdeux}[2]{\ensuremath{\mathfrak{S}_2(#1,#2)}}

\newcounter{longtmp}

\DeclareMathOperator{\diam}{diam}
\DeclareMathOperator{\spec}{spec}
\DeclareMathOperator{\rank}{rank}
\DeclareMathOperator{\card}{card}

\usepackage[latin1]{inputenc}
\usepackage[T1]{fontenc}
\usepackage{amsmath}
\usepackage{amssymb}
\usepackage{amsthm}
\usepackage{amssymb}
\usepackage{accents}
\usepackage{tikz}
\usetikzlibrary{calc}
\usetikzlibrary{decorations.markings}
\usepackage{hyperref}

\newtheoremstyle{mytheorem}
{3pt}
{3pt}
{\normalfont}
{0pt}
{\bfseries}
{.}
{.5em}
{}
\theoremstyle{mytheorem}

\newtheorem{thm}{Theorem}[section]
\newtheorem{prop}[thm]{Proposition}
\newtheorem{lemma}[thm]{Lemma}
\newtheorem{cor}[thm]{Corollary}
 
\newtheorem{question}[thm]{Question} 

\renewenvironment{proof}{\par\noindent{\scshape Proof:}}{\qed\vspace{0.1\baselineskip}\newline\noindent}
\newenvironment{proofarg}[1]{\par\noindent{\scshape #1:}}{\qed\vspace{0.1\baselineskip}\newline\noindent}

\newtheorem{defi}[thm]{Definition}

\newtheorem{rmk}[thm]{Remark}

\title{Symmetric cones, the Hilbert and Thompson metrics}
\author{Bosch\'e Aur\'elien}
\address{Institut Fourier, 100 rue des maths, BP 74, 38402 St Martin d'H\`eres cedex, France}

\begin{document}
\maketitle
\begin{abstract}
  Symmetric cones can be endowed with at least two interesting non Riemannian metrics: the Hilbert and the Thompson metrics. It is trivial that the linear maps preserving the cone are isometries for those two metrics. Oddly enough those are not the only isometries in general. We give here a full description of the isometry groups for both the Hilbert and the Thompson metrics using essentially the theory of euclidean Jordan algebras. Those results were already proved for the symmetric cone of complexe positive hermitian matrices by L.~Moln{\'a}r in \cite{MR2529894}. In this paper however we do not make any assumption on the symmetric cone under scrutiny (it could be reducible and contain exceptional factors).
\end{abstract}

\section{Preliminaries}
\label{sec:introdeux}
A cone is a subset $\mathcal{C}$ of some euclidean space $\R^n$ that is invariant by positive scalar exterior multiplication. A convex cone is a cone that is also a convex subset of $\R^n$. A cone $\mathcal{C}$ is proper (resp. open) if its closure contains no complete line (resp. if it's interior is not empty).  In this paper we deal exclusively with open proper cones and $\mathcal{C}$ will always be such a set. The product of two cones is just the usual product of sets. A cone $\mathcal{C}\in\R^n$ is reducible if $\R^n$ splits orthogonally as the sum of two subspaces $A$ and $B$ each  containing a cone $\mathcal{C}_A$ and $\mathcal{C}_B$ such that $\mathcal{C}$ is the product of $\mathcal{C}_A$ and $\mathcal{C}_B$, i.e. the set of all sums of elements of $\mathcal{C}_A$ and $\mathcal{C}_B$. Otherwise we say that $\mathcal{A}$ is irreducible. To a cone $\mathcal{C}\in\R^n$ we attach the set \auto{\mathcal{C}} of all linear isomorphisms of $\R^n$ that preserve $\mathcal{C}$. This is a group called the automorphism group of $\mathcal{C}$. We associate to a cone $\mathcal{C}$ another cone $\mathcal{C}^*$ called its dual and defined by
\begin{equation*}
  \mathcal{C}^*=\left\{x\in\R^n\mid\forall y\in\mathcal{C},\ \langle x,y\rangle>0\right\}.
\end{equation*}
 We say that a cone is self-dual if it is equal to its dual. A cone is symmetric if and only if its automorphism group acts transitively on it (or equivalently if it acts transitively on the set of rays of $\mathcal{C}$) and if it is self-dual. A pointed cone is a couple $(\mathcal{C},e)$ where $\mathcal{C}$ is a cone and $e\in\mathcal{C}$. It is known that for every pointed symmetric cone $(\mathcal{C},e)$, $\mathcal{C}\subset\R^n$, one can canonically construct a euclidean Jordan structure $J$  on the ambient space $\R^n$, and that reciprocally, when $\R^n$ is endowed with a euclidean Jordan  structure one can canonically construct a pointed symmetric cone $(\mathcal{C},e)$. Those two operations are inverse to each other, and a cone is irreducible if and only if the Jordan algebra associated to it (after choosing a base point) is simple (and this does not depend on the choice of the base point). We recommend \cite{faraut1994analysis} for the general theory of Jordan algebras and symmetric cones. In this paper, unless otherwise stated, $J$ will always denote a euclidean Jordan algebra and $(\mathcal{C},e)$ the symmetric cone associated to it. Remark that $e$ is then the identity of $J$.

Let now $\mathcal{C}$ be any proper open convex cone of $\R^n$. For $P$, $Q\in\mathcal{C}$ we define $\M{P}{Q}=\inf\left\{\,t>0\mid tQ-P\in\mathcal{C}\,\right\}$, and then
\begin{align*}
\dthom{P}{Q}&=\log\max\left(\M{P}{Q},\M{Q}{P}\right),\\
\dhil{P}{Q}&=\log \M{P}{Q}\M{Q}{P}.
\end{align*}
Then $\dthom{\cdot}{\cdot}$ is a metric on $\mathcal{C}$ whereas $\dhil{\cdot}{\cdot}$ is only a pseudo-metric (i.e. it is not definite) on $\mathcal{C}$. Since the condition $\dhil{P}{Q}=0$ and $Q=\lambda P$ for some $\lambda>0$ are equivalent the pseudo-metric $\dhil{\cdot}{\cdot}$ induces a metric on the set of rays through $\mathcal{C}$ i.e. on the projectification of $\mathcal{C}$.

A euclidean Jordan algebra $J$ is a finite dimensional linear space  endowed with a (not necessarily alternative) bilinear commutative product such that for all $(a,b)\in V$ we have $a\cdot(b \cdot a^2)=(a\cdot b)\cdot a^2$ and such that $a^2+b^2=0$ implies $a=b=0$. Such an an algebra is always unital and we shall denote $e$ its unit. Hence by assumption a euclidean Jordan algebra is commutative but associativity fails in general. This failure of associativity in turn creates some ``non-cummutativity'' effects (this is certainly the reason why Jordan investigated those algebras for their possible use in quantum theories). For example we define a center which might very well not be trivial
\begin{defi}
  The center of a Jordan algebra $J$ is the subalgebra consisting of all elements $x\in J$ satisfying
  \begin{equation*}
    \forall a,b\in J,\ x(ab)=(xa)b.
  \end{equation*}
\end{defi}
The set of squares of a euclidean Jordan algebra defines a closed proper cone and its interior, the connected component of the unit in the set of invertible elements of the algebra, is a symmetric cone. An element $p\in J$ is an idempotent if $p^2=p$. The map $p\mapsto e-p$ is a bijection of the set of idempotents of $J$. The image of $p$ under this map will be written $p^\prime$. Two idempotents $p$ and $q$ are orthogonal if $pq=0$. An idempotent is primitive if it cannot be written as the sum of two orthogonal elements.
\begin{defi}
  A Jordan frame is a family $(p_i)_{i=1}^r$ of mutually orthogonal primitive idempotents such that $\sum_{1\leq i\leq r}p_i=e$. The cardinality $r$ of a Jordan frame is independent of the Jordan frame and is called the rank of the algebra.  
\end{defi}
\begin{prop}
  To each $x\in J$ is associated a Jordan frame $(p_i)_{i=1}^r$ and a family of real numbers $(\lambda_i)_{i=1}^r$ such  that $x=\sum_{1\leq i\leq r}\lambda_i p_i$. The $\lambda_i$ only depend (up to reordering) on $x$. We say that $x$ is regular if the $\lambda_i$ are distinct. Under those circumstances the idempotents $p_i$ are also well defined (up to reordering, the reordering being the same as the one alluded to for the $\lambda_i$). 
\end{prop}
Let us now give some notations
\begin{defi}
  If $x=\sum_{1\leq i\leq r}\lambda_i p_i$ is the spectral decomposition of $x\in J$ then we define $\exp(x)=\sum_{1\leq i\leq r}\exp(\lambda_i) p_i$, $\trace{x}=\sum_{1\leq i\leq r}\lambda_i$, $\|x\|=\sup_{1\leq i\leq r}|\lambda_i|$, and $|x|_{\sigma}=\sup_{i,j}|\lambda_i-\lambda_j|$. The $\lambda_i$ are the eigenvalues of $x$ and the set of the eigenvalues is called the spectrum of $x$, noted $\spec{x}$. The spectral norm of $x$ is $\|x\|$. We will also call it the JB-norm.
\end{defi}
\begin{defi}
  The set of elements of a euclidean Jordan algebra $J$ with positive eigenvalues is equal to the image of $J$ by the exponential. This is by definition the (symmetric) cone associated to the Jordan algebra $J$. 
\end{defi}
We are now ready to define a scalar product on $J$
\begin{defi}
  We define $(x,y)=\trace{xy}$ for $x$, $y\in J$.
\end{defi}
\begin{prop}
  The scalar product $(\cdot,\cdot)$ is associative, i.e. satisfies
  \begin{equation*}
    \forall x,y\in J,\ (xz,y)=(x,zy).
  \end{equation*}
\end{prop}
\begin{defi}
  To each $x\in J$, the linear endomorphism $y\mapsto xy$ of $J$ is noted $L(x)$. To each such $x$ we associate another linear endomorphism $\quadr{x}$ of $J$ called the quadratic representation of $x$ and defined by $\quadr{x}=2L^2(x)-L(x^2)$. If $x\in\mathcal{C}$ then $\quadr{x}$ is a positive definite operator for the natural scalar product.
\end{defi}
\begin{rmk}
  Hence $x\in J$ lies in the center of $J$ if and only if $L(x)$ commutes with every $L(y)$, $y\in J$.
\end{rmk}
\begin{defi}
  A Jordan algebra is simple if its only strict ideal is the trivial ideal. A Jordan algebra is semi-simple if it is a direct sum of simple Jordan algebras.
\end{defi}
The following Proposition can be found for example in \cite{faraut1994analysis} for example
\begin{prop}
  Every semi-simple Jordan algebra decomposes uniquely as the direct sum of simple Jordan algebras. Euclidean Jordan algebras are semi-simple.
\end{prop}
 Let us remind the reader that the simple euclidean Jordan algebras have been classified.

Let us turn our attention to the isometry group of the Thompson metric and of the Hilbert semi-metric. Let us begin with the easy
\begin{prop}
  The automorphism group $\auto{\mathcal{C}}$ of a convex proper open cone $\mathcal{C}$ acts isometrically on $\mathcal{C}$ for both the Thompson metric and the Hilbert semi-metric. Indeed if $g\in\auto{\mathcal{C}}$ and $(P,Q)\in\mathcal{C}$ then $\M{P}{Q}=\M{g(P)}{g(Q)}$. This action is faithful in the case of the Thompson metric, and induces an isometric action on the set of rays through $\mathcal{C}$ with kernel the subgroup of positive dilatation $\{\,\lambda I_n,\ \lambda>0\,\}$ in the case of the Hilbert pseudo-metric.
\end{prop}
We will see that this group is not always the full isometry group $\isom{\mathcal{C}}$ for the chosen metric (Hilbert or Thompson), but that it is always a subgroup of finite index of it. 

Every cone associated to a euclidean Jordan algebra also carries a Riemannian symmetric structure of negative Ricci curvature. For convenience we remind the definition of its first fondamental form
\begin{defi}
  The scalar product $\langle \cdot,\cdot\rangle_{x}$ at $x\in\mathcal{C}$ is given by (as usual we identify the tangent space at $x$ with the vector space obtained by forgetting the algebra structure of $J$)
  \begin{align*}
    \forall u,v\in J,\ \langle u,v\rangle_{x}=(\quadr{x}^{-1}u,v)=(\quadr{x}^{-1/2}u,\quadr{x}^{-1/2}v).
  \end{align*}
  This Riemannian structure is complete (locally symmetric manifolds are always complete) and we note $\driem{\cdot}{\cdot}$ the associated metric. We write $i_x$ for the geodesic inversion at $x\in\mathcal{C}$ for this Riemannian structure.
\end{defi}
This Riemannian structure is non-positively curved and simply-connected. In other words it is a Hadamard manifold. Consequently there is exactly one geodesic joining any two points and one can hence define the midpoint of such a pair. Remark that this is in sharp contrast to what happens for both the Hilbert and the Thompson metric. Indeed, putting aside the trivial case where $\mathcal{C}$ is reduced to a half line, the Thompson metric is never locally uniquely geodesic, and the Hilbert metric is locally uniquely geodesic if and only if it is isometric to the model space of constant curvature $-1$.
\begin{defi}
  The midpoint of two points $a$ and $b$ for the Riemannian metric associated to $J$ is written $a\# b$. In fact we have
  \begin{align*}
  a\# b=\quadr{a^{1/2}}\left((\quadr{a^{-1/2}}b)^{1/2}\right)
  \end{align*}
  and $a\# b$  is the unique solution of  $\quadr{x}(a^{-1})=b$. 
\end{defi}
For every $u\in J$, $\det(\exp(u))=\exp(\trace{u})$. But then the set of points $a\in\mathcal{C}$ such that $\det{a}=1$ is the image by the exponential of the kernel of the linear form $u\mapsto \Trace{u}$ ().
\begin{defi}
  Let $J_0$ be the set of $u\in J$ satisfying $\Trace{u}=0$. Then $J_0$ is a linear subspace of $J$ but not a subalgebra in general.
\end{defi}
\begin{defi}
  Let $\mathcal{C}_0$ be the image of $J_0$ by the exponential map. Then $\mathcal{C}_0$ is the set of all $a\in\mathcal{C}$ such that $\det{a}=1$. It is also also a global section of the projectification of $\mathcal{C}$ (because the determinant is positive on $\mathcal{C}$)
\end{defi}
Let us remind the expression of the Riemannian geodesics starting at $e$
\begin{prop}
  The constant speed geodesics starting at $e$ are exactly the curves of the form $t\mapsto\exp(tu)$ with $u\in J^*$. The speed of this geodesic it precisely $\Trace{u^2}^{1/2}$, i.e. the square root of the sum of the squares of the eigenvalues of $u$ with multiplicities.
\end{prop}
We already introduced the geometric mean in general euclidean Jordan algebras. We now introduce a new mean called the spectral mean (see \cite{kim2006jordan} )
\begin{defi}
 The spectral mean $\specmean{a}{b}$ of $(a,b)\in\mathcal{C}$ is $\quadr{a^{-1}\# b}^{1/2}a$. It is the unique solution in $J$ of the equation
 \begin{equation*}
   (a^{-1}\# b)^{1/2}=a^{-1}\# x.
 \end{equation*}
\end{defi}
Let us introduce a new concept before stating the next Proposition
\begin{defi}
  Two elements $a$ and $b$ of $J$ are simultaneously diagonalisable if they are diagonal in the same Jordan frame i.e. if for some Jordan frame $(e_i)_{1\leq i\leq r}$ there exists $(\lambda_i)_{1\leq i\leq r}\in\R^r$ and $(\mu_i)_{1\leq i\leq r}\in\R^r$ such that $a=\sum_{1\leq i\leq r}\lambda_i e_i$ and $b=\sum_{1\leq i\leq r}\mu_i e_i$.
\end{defi}
\begin{rmk}
  Two simultaneously diagonalisable primitive idempotents are obviously either equal or orthogonal.
\end{rmk}
The following proposition is proved in \cite{kim2006jordan}
\begin{prop}
  For $a$, $b\in J$ the following three conditions are equivalent
  \begin{itemize}
  \item $a$ and $b$ are simultaneously diagonalisable,
  \item $\exp(a)$ and $\exp(b)$ are simultaneously diagonalisable,
  \item the geometric mean and the spectral mean of $\exp(a)$ and $\exp(b)$ are equal.
  \end{itemize}
\end{prop}
\section{More on the Hilbert and Thompson metrics}

In \cite{MR1636922} and \cite{MR2529894} the expression of the Hilbert and Thompson metrics for the simple euclidean Jordan algebra of complex hermitian matrices is derived. The computations work in full generality and we include a proof for the ease of the reader
\begin{prop}
  Let us consider the cone associated to a euclidean Jordan algebra $J$. Its associated Hilbert metric $d_H$ and Thompson metric $d_T$ are given by
  \begin{align*}
    \dhil{a}{b}&=\diam\log\spec\left(\quadr{a^{-1/2}}b\right),& \dthom{a}{b}&=\|\log\left(\quadr{a^{-1/2}}b\right)\|,
  \end{align*}
where $\|\cdot\|$ is the spectral norm (that is the JB-norm).
\end{prop}
\begin{proof}
Remember that if $a\in\mathcal{C}$ then $P(a)$ preserves $\mathcal{C}$, so that
    \begin{align*}
    \log \M{a}{b}&=\log\inf\{\,t>0\mid tb-a\in\pp\,\}\\
    &=\log\inf\{\,t>0\mid t-\quadr{b^{-1/2}}a\in\pp\,\}\\
    &=\log\sup\spec{\quadr{b^{-1/2}}a},
  \end{align*}
  and similarly 
  \begin{align*}
    \log \M{b}{a}&=\log\inf\{\,t>0\mid ta-b\in\pp\,\}\\
    &=\log 1/\sup\{\,t>0\mid a-tb\in\pp\,\}\\
    &=\log 1/\sup\{\,t>0\mid \quadr{b^{-1/2}}(a)-t\in\pp\,\}\\
    &=-\log \inf\spec\quadr{b^{-1/2}}(a).
  \end{align*}
The proposition is a direct consequence of those computations.
\end{proof}
\begin{prop}
  The constant speed geodesics for the Riemannian metric  on $\mathcal{C}$ (resp. on $\mathcal{C}_0$) are constant speed geodesics for the Thompson metric (resp. for the Hilbert metric). Consequently the Riemannian midpoints are also midpoints for those two other metrics.
\end{prop}
\begin{proof}
  Since the isometry group acts transitively for the three metrics we can concentrate on the Riemannian geodesics emanating from the identity. If $c:t\mapsto\exp(tu)$, $u\in J$, is a such a geodesic then for every $s$, $t\in\R$.
\begin{align*}
  \dhil{\exp(su)}{\exp(tu)}&=\diam{\log{\spec{\exp((t-s)u)}}}\\
  &=\diam{\spec{(t-s)u)}}\\
  &=|t-s|\diam{\spec{u}}.
\end{align*}
Similarly
\begin{align*}
  \dthom{\exp(su)}{\exp(tu)}&=\|\log{\exp((t-s)u)}\|\\
  &=\|(t-s)u)\|\\
  &=|t-s|\|u\|.
\end{align*}
\end{proof}
The Riemannian geodesics are not only geodesics for those two other metrics but even play a special role among all the geodesics as we shall see in the next section.  The reason why it is so is basically the following lemma
\begin{lemma}
  The Riemannian geodesic inversions are isometries for both the Thompson and the Hilbert metrics.   
\end{lemma}
\begin{proof}
  Since we already found a transitive isometry common to the three metrics it is enough to prove that the geodesic inversion at $e$ is an isometry. But the Riemannian inversion at the identity is just the algebra inversion $a\mapsto a^{-1}$. But one easily proves that $\M{a}{b}=\M{b^{-1}}{a^{-1}}$ for $a$ and $b$ in $\mathcal{C}$. 
\end{proof}
We now carry out a construction that we shall need later. For $\lambda>0$ let us define
\begin{equation*}
  \begin{array}{rcl}
    \Phi_{\lambda}:J&\rightarrow&\mathcal{C}\\
    u&\mapsto&\exp(\lambda u).
  \end{array}
\end{equation*}
Then $\Phi_{\lambda}$ is a homeomorphism. Using $\Phi_{\lambda}$ it is possible to push back the metric $d_T/\lambda$ from $\mathcal{C}$ to $J$. We call $\dthomlam{\lambda}{\cdot}{\cdot}$ this metric. Using the double restriction $(\Phi_{\lambda})_{|J_0}^{|\mathcal{C}_0}$ we can do the same with the Hilbert metric and construct a metric $\dhillam{\lambda}{\cdot}{\cdot}$ on $J_0$. It so happens that those metrics converge to norms when $\lambda$ converges to $0$. To prove this we will need the following Lemmas
\begin{lemma}
For $(u,v)\in J^2$ we have
  \begin{align*}
    \quadr{\exp(-tu/2)}\exp(tv)=e+t(v-u)+o(t).
  \end{align*}
\end{lemma}
\begin{proof}
  The map $(a,b)\mapsto P(a)b$ is differentiable and its differential at $(e,e)$ is $(x,y)\mapsto 2P(e,x)e+P(e)y=2x+y$. But then the differential of $t\mapsto P(\exp(-tu/2))\exp(tv)$ at $0$ is $2(-u/2)+v=v-u$. Since $P(e)e=e$ the Lemma is proved.
\end{proof}
\begin{lemma}
  The spectrum is continuous on any Jordan algebra.
\end{lemma}
\begin{proof}
  The characteristic polynomial of $a\in J$ is continuous on $a$ and the roots of a polynomial depend continuously on the polynomial.
\end{proof}
\begin{prop}
The metrics $\dthomlam{\lambda}{\cdot}{\cdot}$ and $\dhillam{\lambda}{\cdot}{\cdot}$ converge when $\lambda$ converges to $0$, and the limit metrics are both given by a norm. For $(u,v)\in J^2$ and $(u_0,v_0)\in J_0^2$ we have
  \begin{align*}
    \lim_{\lambda\rightarrow0}\dthomlam{\lambda}{u}{v}&=\|v-u\|, & \lim_{\lambda\rightarrow0}\dhillam{\lambda}{u_0}{v_0}\
=\|u_0-v_0\|_{\sigma}.
  \end{align*}
\end{prop}
\begin{proof}
    For $\lambda>0$ we have
  \begin{align*}
    \dthomlam{\lambda}{u}{v}&=\dthom{\exp(\lambda u)}{\exp(\lambda v)}/\lambda\\
    &=\|\log\left(\quadr{\exp(-\lambda u/2)}\exp(\lambda v)\right)\|/\lambda\\
    &=\|\log\left(e+\lambda(v-u) +o(\lambda)\right)\|/\lambda\\
    &=\|v-u\| +o(1),
  \end{align*}
and 
  \begin{align*}
    \dhillam{\lambda}{u_0}{v_0}&=\dhil{\exp(\lambda u_0)}{\exp(\lambda v_0)}/\lambda\\
    &=\diam\log\spec\left(\quadr{\exp(\lambda -u_0/2)}\exp(\lambda v_0)\right)/\lambda\\
    &=\diam\spec\log\left(e+\lambda(v_0-u_0) +o(\lambda)\right)/\lambda\\
    &=\diam\spec\left(\lambda(v_0-u_0) +o(\lambda)\right)/\lambda\\
    &=\diam\spec(v_0-u_0)+o(1).
  \end{align*}
\end{proof}

\section{Isometries fixing the identity}
We begin by the following fundamental result which was already used in the space case of hermitian definite positive complexe matrices in \cite{MR2529894}
\begin{prop}
Every isometry $g$ of the Thompson or the Hilbert metric preserves the Riemannian midpoints, that is satisfies $g(a\# b)=g(a)\# g(b)$.
\end{prop}
\begin{proof}
  See the Lemma in \cite{MR2529894} and how it is applied to show that isometries for the Hilbert and Thompson preserve the Riemannian midpoints (since the proof is exactly the same as in \cite{MR2529894} we do not duplicate it here).
\end{proof}
From this we infer
\begin{prop}
  Every isometry $g$ of the Thompson or the Hilbert metric preserves the Riemannian geodesics and in particular Riemannnian geodesic lines.
\end{prop}
\begin{proof}
  An isometry for any of those metrics must be a homeomorphism of  the underlying symmetric space because the Thompson and the Hilbert metrics generate the topology of the underlying manifold. If $[a,b]$ is a compact geodesic then if we put $M_0(a,b)=\{a,b\}$ and define inductively $M_{n+1}(a,b)$ for $n>0$ to be the union of $M_n(a,b)$ and the midpoints of pairs of points of $M_n(a,b)$ then $M(a,b)=\cup_{n\geq0}M_n$ is dense in $[a,b]$. But since $g$ preserves the midpoints we must have $g(M(a,b))=M(g(a),g(b))$ and by density $g([a,b])=[g(a),g(b)]$.
\end{proof}
Assume now that $g$ is an isometry for either the Thompson or the Hilbert metric and that $g$ fixes $e$. Let $d$ be the metric for which $g$ is an isometry. If $\lambda>0$ then $g$ must be an isometry for $d/\lambda$ too. If $d$ is the Thompson metric, let $g_\lambda$ be the push-back by $\Phi_\lambda$ of $g$ and $d_\lambda$ the push-back of the metric $d/\lambda$ by the same homeomorphism. If $d$ is the Hilbert metric, replace $\Phi_\lambda$ by its double restriction $(\Phi_\lambda)_{|J_0}^{|\mathcal{C}_0}$ and proceed similarly. Then obviously $g_\lambda$ must be an isometry of $d_\lambda$. Let us compute $g_\lambda$
\begin{defi}
  If $g$ is an isometry of $d_T$ (resp. of $d_H$) fixing the identity then $g$ sends a Riemannian constant speed geodesic $t\mapsto\exp(tu)$, $u\in J$ (resp. $u\in J_0$), to another Riemannian constant speed geodesic and hence there exists a well-defined $v\in J$ (resp. $v\in J_0$) such that $g(\exp(tu))=\exp(tv)$ for every $t\in\R$. We write $g_*(u)$ for this $v$. Obviously $g_*$ is homogeneous of degree one and for every $u\in J$ (resp. $u\in J_0$) and $t\in\R$ we have
  \begin{equation*}
    g(\exp(tu))=\exp(tg_*(u)).
  \end{equation*}
\end{defi}
We have
\begin{prop}
$g_\lambda$ is constant equal to $g_*$. In particular the $g_\lambda$ converge to $g_*$ and $g_*$ is a surjective isometry of the limit norm $\lim_{\lambda\rightarrow0}d_\lambda$.
\end{prop}
\begin{proof}
  \begin{align*}
      g_\lambda(u)=\frac{1}{\lambda}\log g(\exp{\lambda u})=\frac{1}{\lambda}\log\exp{\lambda g_*(u)}=g_*(u).
  \end{align*}
$g_*$ must be surjective because so is $g$. It is also a consequence of the fact that it is an isometry of a finite dimensional normed space.
\end{proof}
\begin{prop}
  $g_*$ is a linear isomorphism of $J$ if $d$ is the Thompson metric and of $J_0$ if $d$ is the Hilbert metric.
\end{prop}
\begin{proof}
  Direct consequence of the Mazur-Ulam Theorem.
\end{proof}
We proved that isometries for the Thompson or the Hilbert metric are well-behaved with respect to the geometric mean. In fact they also behave nicely with respect to the spectral mean as the following Proposition shows
\begin{prop}
 Every isometry $g$ of the Thompson or the Hilbert metric preserves the spectral mean i.e. satisfy $g(\specmean{a}{b})=\specmean{g(a)}{g(b)}$ for every $(a,b)\in\mathcal{C}^2$.
\end{prop}
\begin{proof}
 We can assume that $g$ fixes the identity. Since $g$ preserves midpoints it must preserve the inversion and the square root. But since the spectral mean of $a$ and $b\in\mathcal{C}$ is the only solution $x$ in $\mathcal{C}$ of $(a^{-1}\# b)^{1/2}=a^{-1}\# x$, $g$ must also preserve it.
\end{proof}
\section{Case of the Thompson metric}
\begin{defi}
  A symmetry of an algebra $\mathcal{A}$ is an element $s\in\mathcal{A}$ such that $s^2=1$. It is called central if it lies in the center of the algebra $\mathcal{A}$.
\end{defi}
The following is proved in \cite{isidro1995isometries}. Let us recall that a euclidean algebra endowed with the already defined $JB$-norm is a (finite dimensional) $JB$-algebra.
\begin{prop}
  The isometries of a (not necessarily simple) JB-algebra are exactly the maps $x\mapsto b\cdot\Phi(x)$ where $b$ is a central symmetry and $\Phi$ is a Jordan isomorphism. For unital isometries (i.e. preserving the unit $e$) we have $b=e$. 
\end{prop}
Now let $g$ be any isometry for the Thompson metric. After composing $g$ on the right by some element $h$ of the transitive isometry group $\auto{J}$ we get an isometry fixing the identity $e$. From now on we hence assume that $g$ fixes $e$. We proved in the preceding section that $g_*$ is then an isometry of the JB-algebra $J$. Assume first that it fixes the identity. Then according to the proposition above it must be an algebra isomorphism of $J$, and so
\begin{align*}
  \forall u\in\mathcal{C},\ g(\exp(u))=\exp(g_*(u))=g_*(\exp(u)).
\end{align*}
But then $g$ is the restriction to $\mathcal{C}$ of an algebra isomorphism of $J$. Let us come back to the general case. Then we can fix a central symmetry $b$ and an algebra isomorphism $\Phi$ such that $g_*(x)=b\Phi(x)$ for all $x\in J$.
The following Lemma is certainly well known but since we did not find any proof in the existing litterature we include one that does not use the classification of simple euclidean Jordan algebras
\begin{lemma}\label{lem:center}
  The center of a simple euclidean algebra is $\R e$ where $e$ is the unit element.
\end{lemma}
\begin{proof}
  Let $z$ be in the center and let us write  $( x,y)$ for the associative bilinear form $\Trace{xy}$. Then the bilinear form $(x,y)\mapsto\langle zx,y\rangle$ is clearly also an associative bilinear form. But then according to the Proposition \setcounter{longtmp}{3}\Roman{longtmp}.$4$.$1$. of \cite{faraut1994analysis} this form must be a multiple of the original one, i.e. for some $\lambda\in\R$ we have
  \begin{align*}
    \forall x,y\in J,\ ( zx,y)=\lambda ( x,y).
  \end{align*}
Choosing $x=e$ and $y$ arbitrary we obtain that $z-\lambda e$ is in the radical of $J$. Since the radical is reduced to $\{0\}$ by assumption we must have $z=\lambda e$.
\end{proof}
\begin{lemma}
  Consider the decomposition $J=J_1\times\cdots J_n$ into simple euclidean algebras, and let $e_i$ be the multiplicative unit of $J_i$. Then the central symmetries are exactly the $\sum_{i=1}^n \epsilon_i e_i$ where $\epsilon_i\in\{\pm1\}$.
\end{lemma}
\begin{proof}
  The elements of the form $\sum_{i=1}^n \epsilon_i e_i$ with $\epsilon_i\in\{\pm1\}$ are obviously central symmetries. Reciprocally if $u=\sum_{i=1}^n u_i$ is a central symmetry of $J$ where the $u_i$ lie in $J_i$ then each $u_i$ must be a central symmetry of $J_i$. According to the Lemma \ref{lem:center} the center of $J_i$ is $\R e_i$, and so the $u_i$ must be equal to either $e_i$ or $-e_i$.
\end{proof}
$\Phi$ being an algebra isomorphism must permute isometric simple factors. Hence after composing on the right by the corresponding permutation algebra isomorphism $\sigma$ we can assume that $\Phi$, and hence $g_*$, preserves each irreducible factor (remark that $\sigma\in\auto{\mathcal{C}}$). We just proved
\begin{prop}
  Let $g$ be an isometry for the Thompson metric. Then after composing $g_*$ on the right by some algebra isomorphism $\sigma^*$ we get $x\mapsto b x$ for some central symmetry $b$. For $g$ this means that after composing by some $\sigma\in\auto{\mathcal{C}}$ we get a map $a=\sum_{1\leq i\leq n}a_i\mapsto\sum_{1\leq i\leq n}a_i^{\epsilon_i}$ for some $\epsilon_i\in\{\pm1\}$.
\end{prop}
\begin{rmk}
  The map $a_i\mapsto a_i^{-1}$ is just the geodesic inversion at $e_i$ of the symmetric space associated to the simple factor $J_i$ with unit $e_i$. 
\end{rmk}
\begin{rmk}\label{rmk:index}
  Let $n$ be the number of distinct isomorphism classes of the simple factors of $J$. Let us order those isomorphism classes arbitrarily from $1$ to $n$. Suppose that there are $k_i\geq1$ distinct simple factors of $J$ that represent the class numbered $i$. Then the automorphism group is easily seen to have index at most
  \begin{equation*}
    \prod_{1\leq i\leq n}\sum_{0\leq j\leq k_i} ( k_i+1 ) = k+n,
  \end{equation*}
  where $k$ is the number of simple factors of $J$. Indeed this is a upper bound on the number of central symmetries with disjoint orbits under permutation of isometric simple factors of $J$. 
\end{rmk}
We will see in the next section  that there are less isometries for the Hilbert metric as soon as the algebra is not simple. This comes from the fact that products of Thompson isometries are again Thompson isometries, whereas the analogous statement does not hold for Hilbert metrics. Indeed, for the Thompson metric, we have the
\begin{prop}
 Let $\mathcal{C}$ be a product of cones $\mathcal{C}_i$, $1\leq i\leq n$. Let $d$ (resp. $d_i$) be the Thompson metric associated with $\mathcal{C}$ (resp. associated with $\mathcal{C}_i$). Then
 \begin{equation*}
   d=\sup_{1\leq i\leq n} d_i.
 \end{equation*}
\end{prop}
\begin{proof}
Direct consequence of the following computations
  \begin{align*}
    \MC{C}{P}{Q}&=\inf\{\,t>0\mid tQ-P\in \mathcal{C}\,\}\\
    &=\inf\{\,t>0\mid\forall 1\leq i\leq n,\  tQ_i-P_i\in \mathcal{C}_i\,\}\\
    &=\sup_{1\leq i\leq n}\inf\{\,t>0\mid tQ_i-P_i\in \mathcal{C}_i\,\}\\
    &=\sup_{1\leq i\leq n}\MC{C_i}{P_i}{Q_i}.
  \end{align*}
\end{proof}
\begin{rmk}
It follows from this Proposition that the map $a=\sum_{1\leq i\leq n}a_i\mapsto\sum_{1\leq i\leq n}a_i^{\epsilon_i}$ (for some $\epsilon_i\in\{\pm1\}$) is an isometry for the Thompson metric (because geodesic inversions are). From this it follows easily that the index of the automorphism group in the full isometry group is exactly equal to $k+n$ (the notations are those of the remark \ref{rmk:index} ).
\end{rmk}
\section{Case of the Hilbert metric}
This case requires substantially more work than the case of the Thompson metric. The reason is that though we associated some linear map $h^*$ to every isometry fixing the origin $e$, this map is not defined on $J$ but on the hyperplane $J_0$ of $J$. Moreover, even though $h^*$ is an isometry for some norm, this norm is not the restriction to $J_0$ of the JB-norm of $J$. This problem was already encountered in \cite{molnar2003linear} and we will begin our proof likewise. However the Jordan algebra considered in \cite{molnar2003linear} is both simple and exceptional so we have to proceed differently. Remark that our proof is not considerably longer than the one in the aforementioned paper.
\begin{defi}
  We note $\bar{J}$ the quotient vector space $J/(\R e)$. The class of $u\in J$ is noted $[u]$.
\end{defi}
$\bar{J}$ is naturally linearly isomorphic to $J_0$ but it is sometimes better to work with $\bar{J}$. We provide $\bar{J}$ with a norm through this identification.
\begin{lemma}\label{lem:eig}
  The lower (resp. upper) eigenvalue of $L(x)$ is equal to that of $x$.
\end{lemma}
\begin{proof}
  It is well known (see \cite{faraut1994analysis} for example) that the Lemma holds when $x$ is an idempotent since then (putting the two trivial cases aside) the eigenvalues of $x$ are $0$ and $1$ and those of $L(x)$ are among $0$, $1/2$ and $1$. The general case follows from this remark and the spectral decomposition Theorem.
\end{proof}
\begin{rmk}
In fact the eigenvalues of $L(x)$ can be deduced from that of $x$. Indeed in the Proposition $2$.$1$ of \cite{kim2006jordan} it is proved that when $J$ is simple the eigenvalues of $L(x)$ are  precisely the $(\lambda_i+\lambda_j)/2$ for $i\neq j$ (where the $\lambda_i$ are the eigenvalues of $x$). The general case follows from this one by splitting $J$ into simple algebras. Indeed if $J=J_1\times\cdots J_n$ is such a splitting and $x=(x_1,\ldots,x_n)\in J$ then the eigenvalues of $x$ are those of the $x_i$ and the eigenvalues of $L(x)$ are those of the $L(x_i)$.
\end{rmk}
\begin{cor}\label{cor:eig}
  If $x\in J$ has eigenvalues contained in $[0,1]$ then, with respect to the usual partial ordering of symmetric operators, $0\leq L(x)\leq Id$ (where $Id$ is teh identity mapping of $J$) i.e.
  \begin{align*}
    \forall y\in J,\ 0\leq(L(x)y,y)\leq \|y\|^2.
  \end{align*}
\end{cor}
\begin{prop}
  The extremal points of the unit ball of $\bar{J}$ are exactly the classes $[p]$ of the non trivial idempotents of $J$.
\end{prop}
\begin{proof}
Adapted from the Lemma~$2$ in \cite{molnar2003linear}.  Let $p$ be an idempotent and let us show that $[p]$ is an extreme point of the unit ball. Let us write $[p]=t[a]+(1-t)[b]$ for some $t\in]0,1[$ and $(a,b)\in J^2$ with $|a|_{\sigma}=|b|_{\sigma}=1$. Hence for some additional $\lambda\in\R$ we have $p=ta+(1-t)b+\lambda e$ and we can always assume that the spectrum of both $a$ and $b$ is contained in $[0,1]$ and that $0$ and $1$ are eigenvalues of both $a$ and $b$. Hence according to the Corollary \ref{cor:eig} we have $0\leq L(a)\leq Id$ and $0\leq L(b)\leq Id$, from which we deduce that
  \begin{align*}
    ( L(p)p,p)=1&=t( L(a)p,p)+(1-t)( L(b)p,p) +\lambda\|p\|^2\\
    &\leq t\|p\|^2+(1-t)\|p\|^2+\lambda\|p\|^2\\
    &\leq1+\lambda\|p\|^2,
  \end{align*}
and so $\lambda\geq0$. Since $p$ is not trivial, there exists some non trivial idempotent $q\in J$ such that $pq=0$ (take $q=p^\prime=e-p$ for example). But then using the Corollary \ref{cor:eig} again we obtain this time
\begin{equation*}
     ( L(p)q,q)=0=t( L(a)q,q)+(1-t)( L(b)q,q) +\lambda\|q\|^2\geq \lambda\|q\|^2, 
\end{equation*}
so that finally $\lambda=0$ and $p=ta+(1-t)b$. But it is proved in \cite{wright1978isometries} (see also the remark after the Theorem $1$.$1$ in \cite{isidro1995isometries} ) that the projections are the extreme points of the interval $[0,1]$ (i.e. of elements with spectrum contained in $[0,1]$) so that we must have $p=a=b$ and $p$ is indeed an extreme point of the unit ball.

Any class with unit norm of $\bar{J}$ is represented by an element $v\in J$ with sprectum contained in $[0,1]$ and containing both $0$ and $1$. Moreover the class $[v]$ of such an element contains an idempotent if and only if $v$ is itself idempotent. Assume that $v$ is not. Then some of its eigenvalues lie in $]0,1[$ and we can write $v=\lambda_1 p_1+\sum_{2\leq i\leq r}\lambda_i p_i$ with $0<\lambda_1<1$ and $0\leq\lambda_i\leq 1$. But then $v=\alpha d+(1-\alpha) f$ where $\alpha=1-\lambda_1$, $f=\lambda_1/2p_1+\sum_{2\leq i\leq r}\lambda_i p_i$ and $d= (\lambda_1/2+1/2)p_1+\sum_{2\leq i\leq r}\lambda_i p_i$. Obviously $[d]$ and $[f]$ are still in the unit ball of $\bar{J}$ and are distinct so that $[v]$ is not an extreme point of the unit ball.
\end{proof}
\begin{lemma}
  Assume that some element $u\in J$ can be written $u=p+\lambda e$ for some non-trivial idempotent $p$. Then $p$ and $\lambda$ are well defined  and depend continuously on $u$.
\end{lemma}
\begin{proof}
Remark that $p^\prime=e-p$ is a non-trivial idempotent orthogonal to $p$ and that
\begin{align*}
  u=p+\lambda e=p+\lambda(p+(e-p))=(1+\lambda)p+\lambda (e-p).
\end{align*}
Hence the eigenvalues of $u$ are exactly $\lambda$ and $1+\lambda$. So $\lambda$ is the smallest eigenvalue of $u$ and depends continuously on $u$ by continuous dependence of the roots of polynomials. But then $p=u-\lambda e$ also depends continuously on $u$.
\end{proof}
Let $h_*$ be any isometry of $J_0$. Let us extend $h_*$ linearly to $J$ by sending $e$ to itself and write $\hat{h}_*$ for the extended map. Then  $\hat{h}_*$ is still an isometry (on the whole of $J$) for $|\cdot|_{\sigma}$, though $|\cdot|_{\sigma}$ is not a norm anymore (it is degenerate since for example $|e|_{\sigma}=0$). If $p$ is a non trivial idempotent of $J$ then we can write $\hat{h}_*(p)=q+\lambda e$ for some non trivial idempotent q and $\lambda\in\R$, and  $q$ depends continuously on $p$ according to the previous Lemma. Let us write $q=f(p)$ for convenience. Then $f$ is a continuous function from the set $\idem{J}$ of non trivial idempotents of $J$ to itself. The same reasoning with the inverse map $h_*^{-1}$ immediately yields that $f$ is a homeomorphism of $\idem{J}$ and hence preserves the connected components of $\idem{J}$.
\begin{defi}
  For any Jordan algebra $J$ let $\idemrank{k}{J}$ be the set of idempotent of fixed rank $k$, where $k$ is an integer and $J$ is any Jordan algebra. For convenience we put $\idemrank{k}{J}=\{0\}$ if $k\leq0$ and $\idemrank{k}{J}=\{e\}$ if $k\geq\rank{J}$. 
\end{defi}
\begin{lemma}
  Let $J=J_1\times\cdots J_n$ be the decomposition into simple Jordan  algebras of $J$. Then the connected components of $\idem{J}$ are exactly the products $\prod_{i=1}^n \idemrank{k_i}{J_i}$ where $k_i$ is an arbitrary integer. In particular the connected components of $\idem{J_i}$, $1\leq i\leq n$, is the set of idempotents with fixed rank. 
\end{lemma}
\begin{proof}
  One only needs consider the case of a simple Jordan algebra. But then the connected component $K$ of the identity in the isomorphism group of the Jordan algebra $J$ acts transitively on the set of idempotents with fixed rank (See the Proposition\setcounter{longtmp}{4}~\Roman{longtmp}.$3$.$1$, \setcounter{longtmp}{3}(\roman{longtmp}) in \cite{faraut1994analysis} ). The Lemma follows since $K$ is connected.
\end{proof}
\begin{prop}\label{prop:conncomp}
Let $k\in[0,r]$. Then the image by $f$ of a connected component of $\idemrank{k}{J}$ lies in $\idemrank{k}{J}$ or in $\idemrank{r-k}{J}$.
\end{prop}
\begin{lemma}
  Let $m=(m_i)_{1\leq i\leq r}$  be a Jordan frame and $p=\sum_{1\leq i\leq r}\lambda_i m_i$, $\lambda_i\in\{0,1\}$, be a non trivial idempotent which is diagonal in this Jordan frame. Then the set $\Sigm{p}{e}$ of idempotents $q$ which are diagonal in the same Jordan frame $(m_i)_{1\leq i\leq r}$ and that satisfy $\normdiam{p-q}=1$ has cardinality $2^{\rank{p}}+2^{r-\rank{p}}-2$. 
\end{lemma}
\begin{proofarg}{Proof of the Lemma}
The set $\Sigm{p}{m}$ is the disjoint union of the two sets $\Sigmun{p}{m}$ and  $\Sigmdeux{p}{m}$ where $\Sigmun{p}{m}$ consists of the non-zero idempotents $q\in\Sigm{p}{e}$ satisfying $q<p$ (i.e. if $q=\sum_{1\leq i\leq r}\mu_i m_i\in\Sigm{p}{e}$ then $q\neq p$ and $\lambda_i=0$ implies $\mu_i=0$) and $\Sigmdeux{p}{m}=p+\Sigmun{e-p}{m}$ (i.e. if $q=\sum_{1\leq i\leq r}\mu_i m_i\in\Sigmdeux{p}{m}$ then $q\neq p$ and $\mu_i=0$ implies $\lambda_i=0$). Then one has
 \begin{align*}
   \card \Sigmun{p}{m}=2^{k}-1=2^{\rank{p}}-1,
 \end{align*}
and so $\card\Sigmdeux{p}{m}=2^{\rank(e-p)}-1=2^{r-\rank{p}}-1$. The Lemma follows since $\Sigmun{p}{m}$ and $\Sigmdeux{p}{m}$ are disjoint.
\end{proofarg}
\begin{proofarg}{Proof of the Proposition}
  The set of regular elements is an open dense subset of $J$ and by continuity so must be its pre-image by $\hat{h}_*$. The intersection of those two open dense sets is certainly not empty and so we can fix a regular element $x\in J$ such that $\hat{h}_*(x)$ is also regular. Let $m=(m_i)_{1\leq i\leq r}$ (resp. $m^\prime=(m^\prime_i)_{1\leq i\leq r}$) be a Jordan frame in which $x$ is diagonal (resp. $\hat{h}_*(x)$), and remark that by regularity the elements that can be diagonalised in the same frame as $x$ (resp. in the same frame as $\hat{h}_*(x)$) are exactly those that are diagonal in the given frame $m$ (resp. in $m^\prime$). We have already proved that $\hat{h}_*$ preserves simultaneous diagonalisation so elements which are diagonal in the frame $m$ are mapped to elements which are diagonal in $m^\prime$. Moreover the set of idempotents diagonal in the frame $m$ intersects all connected components of $\mathcal{P}(J)$.  Since $\hat{h}_*$ preserves the semi-norm $\normdiam{\cdot}$, so does $f$, and we can now infer that if $0<k<r$
  \begin{align*}
    f(\Sigm{\sum_{1\leq i\leq k}m_i}{m})=\Sigm{f(\sum_{1\leq i\leq k}m_i)}{m^\prime},
  \end{align*}
and in particular those two sets have the same cardinality $2^k+2^{r-k}-2$. But one easily checks that for $(x,y)\in[0,r]^2$,
\begin{align*}
  2^x+2^{r-x}-2=2^y+2^{r-y}-2\Leftrightarrow x=y\quad\textrm{or}\quad x=r-y.
\end{align*}
Indeed, the function $x\mapsto 2^x+2^{r-x}-2$ is symmetric around $r/2$, strictly decreasing on $]-\infty,r/2[$ and strictly increasing on $]r/2,+\infty[$. Hence $f(\sum_{1\leq i\leq k}m_i)$ has rank $k$ or $r-k$ and we are done.
\end{proofarg}

Let now  $J=J_1\times...\times J_s$ be the decomposition of $J$ into simple factors where $J_i$ as unit $e_i$. Let also $r$ (resp. $r_k$) be the rank of $J$ (resp. of $J_k$) and $\mathcal{P}_i(J)$ (resp. $P_i(J_k)$) be the set of idempotents of rank $i$ in $J$ (resp. in $J_k$). Define
\begin{align*}
  Q_1(J_k)=\{e_1\}\times\cdots\{e_{k-1}\}\times \mathcal{P}_{{r_k}-1}(J_k)\times \{e_{k+1}\}\times\cdots\times\{e_s\}.
\end{align*}
 Then the connected components of $\mathcal{P}_1(J)$ (resp. $\mathcal{P}_{r-1}(J)$) are exactly the $P_1(J_k)$ (resp. the $Q_1(J_k)$). Moreover $x\mapsto e-x$ is a diffeomorphism between $P_1(J_k)$ and $Q_1(J_k)$ . Now take $a\in[1,s]$. There are two possibilities for $\mathcal{P}(J_a)$ according to the Proposition \ref{prop:conncomp}

\noindent\textbf{Case one:} $f$ sends $\mathcal{P}_1(J_a)$ onto $\mathcal{P}_1(J_b)$ for some $b$.
Since $h_*$ is trace preserving we must have $h_*(\mathcal{P}_1(J_a))=\mathcal{P}_1(J_b)$ and by linearity $h_*(J_a)=J_b$.

\noindent\textbf{Case two:} $f$ sends $\mathcal{P}_1(J_a)$ onto $Q_1(J_b)$ for some $b$. 
But then if $p\in\mathcal{P}(J_a)$ we have $\Trace{f(p)}=\Trace{e}-\Trace{p}=\Trace{e}-1$ and from $h_*(p)=f(p)+\lambda e$ we then deduce that $\lambda=2/\Trace{e}-1$. The decomposition of $h_*(p)$ corresponding to the splitting $J=J_0\otimes (\R e)$ is then
\begin{align}\label{eq:2}
  h_*(p)=(1/\Trace{e} e+f(p)-e)+1/\Trace{e} e.
\end{align}
Consider now the first factor $J_1$. Assume that we are in the second case above, i.e. that $\mathcal{P}_1(J_1)=Q_1(J_b)$ for some $b\in[1,s]$.Then if we compose $h$ by the inversion $x\mapsto x^{-1}$, $h_*$ is composed by $x\mapsto -x$ and we easily deduce that the extension $\hat{f}_*$ of $h_*$ is replaced  by (see \eqref{eq:2} )
\begin{equation*}
  h_*(p)=-(1/\Trace{e} e+f(p)-e)+1/\Trace{e} e=e-f(p),
\end{equation*}
so that after composing $h$ with the inversion the factor $\mathcal{P}_1(J_1)$ is in the first case above. Since the inversion is an isometry and preserves $J_0$ it follows that we can always assume that $f(\mathcal{P}_1(J_1))=\mathcal{P}_1(J_b)$ for some $b\in[1,s]$.
\begin{lemma}
  Either all the factors are in the first case above or they are all in the second.
\end{lemma}
\begin{proof}
  We can always assume that $s\geq2$. From the previous discussion we can always assume that $J_1$ is such that we have $f(\mathcal{P}_1(J_1))=\mathcal{P}_1(J_b)$ for some $b\in[1,s]$.

Let us prove first that we must have $f(\mathcal{Q}_1(J_1))=\mathcal{Q}_1(J_b)$. If $c\in[1,s]$ is different from $b$ then it is easy to see that the spectrum of $q-p$ is independent of the choice of $q\in\mathcal{Q}_1(J_c)$ and $p\in\mathcal{P}_1(J_b)$ and contains exactly $r-2$ ones and $2$ zeros. But, for $p\in\mathcal{P}_1(J_b)$, $q=e-p\in\mathcal{Q}_1(J_c)$ and the spectrum of $q-p=e-2p$ contains at least one $1$ and one $-1$ (because $s\geq2$). Hence $b$ is the only element of $c\in[1,s]$ $\mathcal{Q}_1(J_b)$ such that we do not have $\|q-p\|_\sigma=1$ for every $q\in\mathcal{Q}_1(J_c)$ and $p\in\mathcal{P}_1(J_b)$. Since $f$ preserves $\|\cdot\|_\sigma$ and $\mathcal{P}_1(J_b)$ we must have $f(\mathcal{Q}_1(J_1))=\mathcal{Q}_1(J_b)$.

Let us assume that $J_2$ is such that $f(\mathcal{P}_1(J_2))=\mathcal{Q}_1(J_c)$ for some $c\in[1,s]$. We must have $c\neq b$ because $f$ is injective. If $p_1\in\mathcal{P}_1(J_1)$ and $p_2\in\mathcal{P}_1(J_2)$ then 
\begin{equation*}
  \|p_1-p_2\|_\sigma=2.
\end{equation*}
However it is easy to see that we must have $\|f(p_1)-f(p_2)\|_\sigma=1$ because $c\neq b$: contradiction.
\end{proof}
Hence after possibly composing $h$ with the inversion we can assume that all the factors are in the first case, i.e. that $f$ preserves the set of irreducible idempotents of each simple factor. Composing again $h$ by the Jordan automorphism that permutes isometric factors we can moreover assume that $h$ preserves each simple factor. We now prove that $h$ must be linear.
\begin{prop}
  $h$ is linear (i.e. is the restriction to $\mathcal{C}$ of a linear isomorphism of $\R^n$).
\end{prop}
\begin{proof}
  We already know that $h_*$ preserves the set of primitive idempotents. Let us show that $h_*$ sends orthogonal primitive idempotents to orthogonal primitive idempotents. But two primitive idempotents are orthogonal if and only if they are simultaneously diagonalisable and distinct. Since $h_*$ preserves simultaneously diagonalisable pairs and is injective it must also preserve pairs of orthogonal primitive idempotents. It follows easily that $h_*$ is a Jordan isomorphism. But then $h_*$ commutes with the exponential and
  \begin{align*}
    \forall x\in J,\ h(\exp x)&=\exp(h_* x)\\
    &=h_*\exp(x).
  \end{align*}
Hence $h$ is the restriction of $h_*$ to the symmetric cone and is linear. 
\end{proof}
\begin{cor}
  The automorphism group is a subgroup of index two or zero in the isometry group of a symmetric cone for the Hilbert metric. The automorphism group is equal to the isometry group only for the Lorentz cones, i.e. only when the underlying Jordan algebra has rank at most two. To be more precise and closer to the spirit of the classification of euclidean simple Jordan algebras, equality between the two groups appear only in the following three cases, the third being an infinite family
  \begin{enumerate}
  \item $\mathcal{C}$ is a half-line (i.e. has rank one),
  \item $\mathcal{C}$ is the positive quadrant of $\R^2$, i.e. the set of points with positive coordinates (this is the direct product of two half-lines),
  \item $n=\dim(J)\geq 3$ and $\mathcal{C}$ is the irreducible Lorentz cone i.e. the set of points $(x_1,\ldots,x_n)$ satisfying $x_1^2>x_2^2+\cdots+x_n^2$.
  \end{enumerate}
\end{cor}
\begin{proof}
First let us show that when the rank is two the geodesic inversion at $e$ is in $\auto{\mathcal{C}}$. But in that case if $(e_1,e_2)$ is a Jordan frame and $u=\lambda_1e_1+\lambda_2e_2$  (for $\lambda_1$, $\lambda_2>0$) we have
\begin{align*}
  u^{-1}=\frac{1}{\lambda_1}e_1+\frac{1}{\lambda_2}e_2&=\frac{1}{\lambda_1\lambda_2}\left(\lambda_2e_1+\lambda_1 e_2\right)\\
  &=\frac{1}{\lambda_1\lambda_2}\left((\lambda_1+\lambda_2-\lambda_1)e_1+(\lambda_1+\lambda_2-\lambda_2)\right)e_2\\
  &=\frac{1}{\lambda_1\lambda_2}\left(\Trace{u}e-u\right),
\end{align*}
which is clearly a projective transformation (if we only consider the restriction of the inversion to the set of $x$ satisfying $\det(x)=1$ then it is even ``linear''). The rank one case is trivial so their only remains to prove that when the rank is at least three the inversion is not projective. This can easily be shown directly but the following argument (due to M.~Crampon) gives more insight into what happens near the boundary of the cone.

First remark that the inversion is homogeneous (of degree $-1$) and so we can work with rays instead of restricting the inversion to  $\mathcal{C}_0$. For $x\in\mathcal{C}$ let $[x]$ be the ray through $x$. Let $e_1$, $e_2$ and $e_3$ be three orthogonal idempotents. Then if $n>0$ the inverse of $u^1_n=ne_1+e_2+e_3$ is $u^{-1}_n=(1/n)e_1+e_2+e_3$. Similarly the inverse of $u^2_n=ne_1+2e_2+e_3$ is $u_n^{-2}=(1/n)e_1+(1/2)e_2+e_3$. But the rays $[u^1_n]$ and $[u^2_n]$ defined by $u^1_n$ and $u^2_n$ both converge to the ray $[e_1]$. If the inversion were a projective map, the rays $[u^{-1}_n]$ and $[u^{-2}_n]$ would also converge to the same ray. But  $[u^{-1}_n]$ converges to $[e_2+e_3]$ and $[u^{-2}_n]$ converges to $[(1/2)e_2+e_3]$. Since those two rays differ the inversion cannot be projective.
\end{proof}
\begin{rmk}
  For the rank two case we could equally have used the fact that the cone is then strictly convex and so the isometry group for the Hilbert metric is reduced to $\auto{\mathcal{C}}/\R$ (see \cite{de1993hilbert} ).
\end{rmk}
\begin{rmk}
  In fact we showed that the inversion does not have a continuous prolongation to the boundary. See the paper of De la Harpe \cite{de1993hilbert} where he investigates the ``blow off'' near the boundary for the simplicial cone of $\R^3$ (this shows that isometries need not admit a prolongation to the boundary of the convex).
\end{rmk}
\begin{question}
  If a map $f:J\rightarrow J$ preserves the semi-norm $\|\cdot\|_\sigma$ and the trace then it acts as the identity on $\R e$ (because $\R e$ is the set of elements with zero $\|\cdot\|_\sigma$ semi-norm and on this set the trace is injective) and its double-restriction to $J_0$ must be linear by the Mazur-Ulam Theorem since on this set the semi-norm $\|\cdot\|_\sigma$ is definite. In fact since the projection onto the second factor of the decomposition $J=J_0\oplus\R e$ is $x\mapsto\trace{x}/\trace{e}e$ and $f$ commutes with this operator, $f$ always preserves the second factor. From this it follows that $f$ must itself be linear and preserve the decomposition (indeed if $x_0=a+\lambda_0 e$ is the decomposition of $x_0$ in the direct sum $J=J_0\oplus\R e$ then $f(x_0)=b+\lambda_0 e$ for some $b\in J_0$ that must itself satisfy  $\|f(x_0)-b\|_{\sigma}=0$ and $\|f(x_0)-f(a)\|_{\sigma}=\|f(a)+\lambda_0 e-f(a)\|_{\sigma}=0$, so that $b=f(a)$). We proved above that if $f$ sends simultaneously diagonalisable pairs to simultaneously diagonalisable pairs then it is either a Jordan isomorphism or becomes one after composing with the linear map that can be written $(x_0,\lambda)\mapsto (-x_0,\lambda)$ in the splitting $J=J_0\oplus\R e$. The question is: do one needs to assume that simultaneously diagonalisable pairs are preserved, or is this always true? We did not find any evidence that the question was already investigated, even in the case of symmetric/hermitian matrices.
\end{question}
\bibliographystyle{plain} 
\bibliography{biblio} 

\begin{thebibliography}{1}

\bibitem{MR1636922}
E.~Andruchow, G.~Corach, and D.~Stojanoff.
\newblock Geometrical significance of {L}{\"o}wner-{H}einz inequality.
\newblock {\em Proc. Amer. Math. Soc.}, 128(4):1031--1037, 2000.

\bibitem{de1993hilbert}
P.~de~la Harpe.
\newblock On hilbert's metric for simplices.
\newblock {\em Geometric group theory}, 1:97--119, 1993.

\bibitem{faraut1994analysis}
J.~Faraut and A.~Kor{\'a}nyi.
\newblock {\em Analysis on symmetric cones}.
\newblock Clarendon Press Oxford, 1994.

\bibitem{isidro1995isometries}
J.M. Isidro and A.R. Palacios.
\newblock Isometries of jb-algebras.
\newblock {\em manuscripta mathematica}, 86(1):337--348, 1995.

\bibitem{kim2006jordan}
J.~Kim and Y.~Lim.
\newblock Jordan automorphic generators of euclidean jordan algebras.
\newblock {\em J. Korean Math. Soc}, 43(3):507--528, 2006.

\bibitem{molnar2003linear}
L.~Molnar and M.~Barczy.
\newblock Linear maps on the space of all bounded observables preserving
  maximal deviation.
\newblock {\em Journal of Functional Analysis}, 205(2):380--400, 2003.

\bibitem{MR2529894}
Lajos Moln{\'a}r.
\newblock Thompson isometries of the space of invertible positive operators.
\newblock {\em Proc. Amer. Math. Soc.}, 137(11):3849--3859, 2009.

\bibitem{wright1978isometries}
JD~Wright and MA~Youngson.
\newblock On isometries of jordan algebras.
\newblock {\em Journal of the London Mathematical Society}, 2(2):339, 1978.

\end{thebibliography}
\end{document}